\documentclass[reqno]{amsart}
\newcommand \datum {March 30, 2021 (for arXiv)}

\usepackage{amssymb,latexsym}
\usepackage{amsmath}
\usepackage{graphicx}
\usepackage{color}
\usepackage{enumerate}
\numberwithin{equation}{section}
\theoremstyle{plain}
 \newtheorem{theorem}{Theorem}[section]
 \newtheorem{lemma}[theorem]{Lemma}
 \newtheorem{proposition}[theorem]{Proposition}
\theoremstyle{definition}
 \newtheorem{definition}[theorem]{Definition}
 \newtheorem{remark}[theorem]{Remark}
\newenvironment{enumeratei}{\begin{enumerate}[\quad\upshape (i)]} {\end{enumerate}}
\newcommand \relE {\mathrel E}
\newcommand \FDthree {\textup{FD}(3)}
\newcommand \lang [1] {\textup{Lng}(#1)}
\newcommand \Pset {P}
\newcommand \then {\Rightarrow}
\newcommand \ordsigma  {{\sigma_{\textup{ord}}}}
\newcommand \grsigma  {{\sigma_{\textup{gr}}}}
\newcommand \fthen  {\mathrel{\models_{\textup{fin}}}}
\newcommand \alg[1] {\mathcal{#1}}
\newcommand \lalg[1] {\mathcal{#1}}
\newcommand \Id[1] {\textup{Dset}(#1)}  
\newcommand \tbf[1] {\textbf{#1}}  
\newcommand \Jir [1] {J(#1)} 
\newcommand \aJir [1] {\alg J(#1)} 

\newcommand \Nnul {\mathbb N_0}
\newcommand \Nplu {\mathbb N^+}
\renewcommand \phi{\varphi}
\DeclareMathOperator{\Con}{Con}
 
\newcommand \set [1]{\{#1\}}
\newcommand \tuple [1] {\langle #1 \rangle}
\newcommand \pair [2] {\tuple{#1,#2}}
\newcommand \LEnl [1] {\textup{LeftEnl}(#1)}
\newcommand \REnl [1] {\textup{RightEnl}(#1)}
\newcommand \maxJir [1] {\textup{Max}(\mathcal J(#1))}
\newcommand \ideal [1] {\mathord{\downarrow} #1}
\newcommand \rjir [2] {\rho_{\textup{ji}}(#1,#2)}
\newcommand \rmjir [2] {\rho_{\textup{mji}}(#1,#2)}
\newcommand \rmset [3] {\rho_{\textup{V}}(#1,#2,\,\,#3)}
\newcommand \rmmset [5] {\rho_{\textup{W}}(#1,#2,#3,#4,\,\,#5)}
\newcommand \rmmm [1] {\rho_{\textup{VW}}(#1)}
\newcommand \rmcyclic [1] {\rho_{\textup{mcycl}}(#1)}

\newcommand \orjir [2] {\rho^{\ast}_{\textup{jir}}(#1,#2)}
\newcommand \ormjir [2] {\rho^{\ast}_{\textup{Mjir}}(#1,#2)}

\newcommand \ormmm [1] {\rho^{\ast}_{\textup{VW}}(#1)}
\newcommand \ormcyclic [1] {\rho^{\ast}_{\textup{mcycl}}(#1)}
\newcommand \defiff {\overset{\textup{def}}{\iff}}
\newcommand \mathand {\mathop{\&}}

\newcommand\red[1]{{\textcolor{red}{#1}}}

\begin{document}

\title[Cyclic congruences and non-finite axiomatizability]{Cyclic congruences of slim semimodular lattices and non-finite axiomatizability of some finite structures}

\author[G.\ Cz\'edli]{G\'abor Cz\'edli}
\email{czedli@math.u-szeged.hu}
\urladdr{http://www.math.u-szeged.hu/~czedli/}
\address{University of Szeged, Bolyai Institute. 
Szeged, Aradi v\'ertan\'uk tere 1, HUNGARY 6720}

\begin{abstract} We give a new proof of the fact that  finite bipartite graphs cannot be axiomatized by finitely many first-order sentences among \emph{finite} graphs. 
(This fact is a consequence of a general theorem proved by L.\ Ham and M.\ Jackson, and the counterpart of this fact for all bipartite graphs in the class of \emph{all} graphs is a well-known consequence of the compactness theorem.) 
Also, to exemplify that our method is applicable in various fields of mathematics, we prove that neither finite simple groups, nor the ordered sets of join-irreducible congruences of slim semimodular lattices can be described by finitely many axioms in the class of \emph{finite} structures.
Since a 2007 result of  G.\ Gr\"atzer and E.\ Knapp, slim semimodular lattices have constituted the most intensively studied part of lattice theory and they have already led to results even in group theory and geometry. In addition to the non-axiomatizability results mentioned above, we present a new property, called Decomposable Cyclic Elements Property, of the congruence lattices of  slim semimodular lattices. 
\end{abstract}

\thanks{This research was supported by the National Research, Development and Innovation Fund of Hungary under funding scheme K 134851.}



%
\dedicatory{Dedicated to B\'ela Cs\'ak\'any on his ninetieth birthday (2022)}


\date{\datum\hfill{\red{\tbf{Hint:} check the author's web page for possible updates}}}

\maketitle

\section{Introduction}\label{sect:intro} 

\subsection{Outline and prerequisites} Section~\ref{sect:intro} is introductory.  
Each of Sections~\ref{sectionmodels} and \ref{sectiongroups} gives an example how to apply the well-known tools of model theory for simple problems on finite axiomatizability among \emph{finite} structures. These two sections as well as Section~\ref{sect:intro}  are
easy to understand for all mathematicians and even their proofs are readable for those who have ever met the concept of ultraproducts. 
Sections~\ref{sectionlatt} and \ref{sect:morelatt} are intended for lattice theorists and rely heavily on a recent paper, Cz\'edli~\cite{czglamps}. The results of Sections~\ref{sectionlatt} and \ref{sect:morelatt} are summarized at the end of (this) Section~\ref{sect:intro}.  

\subsection{Finite model theory}
Finite model theory is a thriving part of mathematics. This is witnessed by, say, the monograph Libkin~\cite{libkin} with its 250 references or by the fact that, at the time of writing, MathSciNet
returns seven matches to the search ``Title=(finite model theory) AND Publication Type=(Books)''. However, 
 the following words of Fagin~\cite[page 4]{fagin} from 1993 are still valid: ``almost none of the key theorems and tools of model theory, such as the completeness theorem and the compactness theorem, apply to finite structures''. 
This could be the reason that the problems of finite model theory are harder than those of (the classical and  unrestricted) model theory.

\subsection{Our goal} 
This paper  deals with the axiomatizability of three different classes of finite structures.  We  prove that none of these three classes can be defined by a finite set of first-order sentences within the class of \emph{finite} structures. 

In case of the first two classes, our goal is to point out that even if the well-known classical methods of classical model theory are usually too weak for finite structures, these methods are still applicable in lucky cases.

The first class consists of all \emph{finite bipartite graphs}. While it is a trivial consequence of the compactness theorem that the class of all (not necessarily finite) bipartite graphs is not finitely axiomatizable, the finite case is a bit more involved.
In Section~\ref{sectionmodels}, we prove that  finite bipartite graphs cannot be finitely axiomatized among finite structures. Although this statement is only a very particular case of Ham and Jackson~\cite[Corollary 4.3]{hamjackson}, the point is that while Ham and Jackson's quite involved proof uses a heavy machinery, our approach is elementary and much simpler. Note that, for the reader's convenience, Section~\ref{sectionmodels} recalls some known facts from finite model theory.

The second class consists of all \emph{finite simple groups}. 
The result that Section~\ref{sectiongroups} presents on this class is quite easy and not at all surprising; the aim of Section~\ref{sectiongroups} is to point out that, sometimes, classical model theory is applicable for \emph{finite} algebras in a variety if  powerful theorems hold in the variety.

In case of the third class, whose definition with an appropriate introduction is postponed to Section~\ref{sectionlatt}, our result is the opposite of what has previously been conjectured. Here we only note that Section~\ref{sectionlatt}, containing one of the two theorems of the  of the paper, belongs to the most intensively studied part of lattice theory.

Related to the third class but dealing with a different one, 
 in Section~\ref{sect:morelatt} 
presents a new property under the name \emph{Decomposable Cyclic Elements Property} of the congruence lattices of slim semimodular lattices; see Theorem~\ref{thm:newproptY} for the statement.
Also, Section~\ref{sect:morelatt}  points out why our approach is not appropriate to prove a  non-axiomatizability result about the lattices of congruences (rather than the posets of join-irreducible congruences) of these lattices.
Notably, the Decomposable Cyclic Elements Property would not have been discovered without the finite model theoretic Section~\ref{sectionlatt}.

The content of Sections~\ref{sectionlatt} and \ref{sect:morelatt},  the lattice theoretical sections, can be summarized as follows.

\begin{enumeratei}
\item The \emph{posets} of join-irreducible congruences of slim semimodular lattices cannot be axiomatized by finitely many axioms among \emph{finite}  structures; see Theorem~\ref{thmnonax}.
\item One of the known properties of the above-mentioned \emph{posets}  cannot be given  by finitely many axioms among \emph{finite} structures; see Remark~\ref{rem:nWhWrlNsW}.
\item All what we currently know about  the \emph{lattices} of congruences of slim semimodular lattices, including (the new) Decomposable Cyclic Elements Property, can be described by finitely many axioms (in fact, even by a single axiom) among \emph{finite} lattices; see Remark~\ref{rem:sNldGsxm}.
\item We do not know whether the \emph{lattices} of congruences of slim semimodular lattices can be axiomatized by finitely many axioms among \emph{finite} lattices; see Remark~\ref{rem:szplMndHcKzgTrs}.
\end{enumeratei}

\section{Non-finite axiomatizability of bipartite graphs}\label{sectionmodels}

We begin this section with recalling some known concepts and facts; they will also be needed in the subsequent sections.
By a \emph{finite signature} we mean a tuple
\begin{equation}
\sigma=\bigl\langle p, q, \tuple{R_1,r_1},\dots \tuple{R_p,r_p}, 
\tuple{F_1,f_1},\dots \tuple{F_q,f_q}\bigr\rangle
\label{eqsignature}
\end{equation}
where $p,q\in\Nnul=\set{0,1,2,\dots}$,  
$R_1,\dots,R_p$ are \emph{relation symbols}, $F_1,\dots,F_q$ are \emph{function symbols}, and these symbols are of arities $r_1$, \dots, $r_p$, $f_1$, \dots, $f_q\in\Nnul$, respectively.  A \emph{structure of type} $\sigma$ or, shortly, a \emph{$\sigma$-structure} is a $(1+p+q)$-tuple
\[\alg A=\tuple{A,R_1^{\alg A},\dots, R_p^{\alg A}, F_1^{\alg A},\dots, F_q^{\alg A}}\text{ or, shortly, }
\tuple{A,R_1,\dots, R_p, F_1,\dots, F_q},
\]
where   $A$, called the \emph{underlying set}, is a nonempty set, $R_i^{\alg A}\subseteq A^{r_i}$ is a relation, and $F_j^{\alg A}\colon A^{f_j}\to A$ is a map for all $i\in\set{1,\dots,p}$ and $j\in\set{1,\dots,q}$. Structures will often be denoted by calligraphic capital letters $\alg A$, $\alg B$, \dots{} while their underlying sets with the corresponding italic capitals $A$, $B$, \dots{} .  (However, sometimes we denote structures simply by their underlying sets.) Note that $x\in \alg A$ will mean that $x\in A$ and similarly for other structures.
In this paper, 
\begin{equation}
\parbox{7.5cm}{the first-order language with equality determined by $\sigma$ will be denoted by $\lang\sigma$.}
\label{eqtxtFslNg}
\end{equation}
So, in addition to the relation symbols and function symbols  occurring in \eqref{eqsignature},
$\lang\sigma$ includes the equality symbol, which is always interpreted as the equality relation. 
To define the (first-order) \emph{consequence relation modulo finiteness}, denoted by $\fthen$, assume that 
$\Phi$ is a set of $\lang\sigma$-sentences and $\mu$ is an $\lang\sigma$-sentence. Then 
\begin{equation}
 \text{$\mu$ is a \emph{consequence of $\Phi$ modulo finiteness}, in notation }  \Phi\fthen \mu,
\end{equation}
if every \emph{finite} $\sigma$-structure that satisfies all sentences belonging to $\Phi$ also satisfies $\mu$.
To see an example, we borrow the  sentence 
\begin{equation}
\lambda_k: \exists x_1\dots\exists x_k\bigwedge_{1\leq i<j\leq n}\neg(x_i=x_j)
\label{eqlambdak}
\end{equation}
for $k\in\Nplu=\set{1,2,3,\dots}$
from Fagin~\cite{fagin}. Here $\neg$ is the negation sign. 
Let $\lambda_{-1}$ be
the (identically false) sentence $\exists x (x=x{} \wedge \neg(x=x))$. Clearly, $\Phi\fthen \lambda_{-1}$ but there is no finite subset $\Phi'$ of $\Phi$ such that $\Phi'\fthen\lambda_{-1}$. This example shows well how big the difference between $\fthen$ and the usual consequence relation (for not necessarily finite structures) is.

\begin{definition}\label{def:grSnTr}
The signature $\grsigma$ of  graphs is the particular cases of 
 \eqref{eqsignature} such that $\pair p q=\pair 1 0$ 
 and we write $E$ and $x\relE y$ instead of $R_1$ and $\pair x y\in R_1^{\alg A}$; the latter means that $\alg A=\tuple{A,E}$ is a  (directed) graph, $x,y\in A$ are vertices, and there is an edge from  $x$ to  $y$. Graphs satisfying the sentence $\forall x\forall y(x\relE y\then y\relE x)$ are \emph{undirected}. An undirected graph $\alg A$ is \emph{bipartite} if 
 there are disjoint nonempty subsets $A_0$ and $A_1$ of $A$ such that $A=A_0\cup A_1$ and 
$E^{\alg A}\subseteq (A_0\times A_1) \cup (A_1\times A_0)$. 
\end{definition}

Using   \eqref{eqtxtFslNg} and Definition~\ref{def:grSnTr},
we present the following statement, which is only a very particular case of Jackson~\cite[Corollary 4.3]{hamjackson}.

\begin{proposition}\label{propositionbipartite} The class of finite bipartite graphs is not finitely axiomatizable modulo finiteness. That is, there exists no finite set $\Sigma$ of\,  $\lang\grsigma$-sentences such that a \emph{finite} graph $\alg A$ is bipartite if and only if each member of\, $\Sigma$ holds in $\alg A$. 
\end{proposition}

Before proving this proposition, we make some comments. Although our aim with Sections~\ref{sectionmodels} and \ref{sectiongroups} is to give \emph{simple proofs}, note that  in most cases, axiomatizability results about  finite structures are proved by the methods offered by the theory of \emph{Ehrenfeucht--Fra\"\i ss\'e Games}. This theory, which is much more complicated than our approach in this paper, goes back to Ehrenfeucht~\cite{ehrenfeucht} and Fra\"\i ss\'e~\cite{fraisse}. For more information on Ehrenfeucht-Fra\"\i ss\'e Games, which will not occur in the rest of the paper, we can recommend the monographs Immerman~\cite{immerman} and Libkin~\cite{libkin}.

Our approach is similar to that of Dittmann~\cite{dittmann} since both approaches use ultraproducts. According to MathSciNet, \cite{dittmann} has not appeared in a journal and it is more or less forgotten. Instead of ultraproducts, one could develop Libkin's idea to use the L\"owenheim--Skolem Theorem; see the proof of
\cite[Proposition 3.3]{libkin}. (While Libkin used his Proposition 3.3  to show how narrow the scope of classical model theory for finite structures is, this paper shows that this scope is less narrow.)

\begin{proof}[Proof of Proposition~\ref{propositionbipartite}] For $2\leq n\in\Nplu$, the \emph{circle of length} $n$ is the graph $\alg C_n$ with base set  $C_n:=\set{0,1,\dots,n-1}$ and $E^{\alg C_n}=\set{\pair x y:  |x-y|\in\set{1,n-1}}$. 
\begin{equation} 
\text{Let }\, J_0:=\set{4,6,8,10,12,14,\dots } \,\text{ and }\, J_1:=\set{3,5,7,9,11,13,\dots }. 
\label{eqwMnJsdef}
\end{equation}
Note that $\alg C_n$ is bipartite for all $n\in J_0$ but $\alg C_n$ is not bipartite if $n\in J_1$. The set of all subsets of $J_i$ will be denoted by $\Pset(J_i)$. 
For $i\in\set{0,1}$, let $U_i$ be a nontrivial ultrafilter over $J_i$;
see, for example, Poizat~\cite{poizat} for this concept. 
What we need here is that  $\emptyset\notin U_i\subseteq\Pset(J_i)$  and $U_i$  contains all cofinite subsets of $J_i$; a subset $X\subseteq J_i$ is \emph{cofinite} if $J_i\setminus X$ is finite. Let $\alg A_i=\tuple{A_i, E}$ be the ultraproduct $\prod_{n\in J_i}\alg C_n/U_i$. 
We know from  Frayne, Morel, and Scott~\cite{frayneatal} or from Keisler~\cite{keisler} that
\begin{equation}
\parbox{10.0cm}{an ultraproduct of finite structures modulo a nontrivial ultrafilter is either finite, or it has at least continuum many elements.}
\label{eqpbxfMsKtTsl}
\end{equation}
For $i\in\set{0,1}$, $k\in \Nplu$, and $\lambda_k$ defined in  \eqref{eqlambdak}, 
$\{n\in J_i: \lambda_k$ holds in $\alg C_n\}$ is a cofinite set, whereby it belongs to $U_i$. 
Hence, $\lambda_k$ holds in $\alg A_i$ by Lo\'s's Theorem; see, for example, Theorem 4.3 in Poizat~\cite{poizat}. Since this is true for all $k\in\Nplu$,  $\alg A_i$ is not finite. Also, it has at most continuum many elements since the cardinality of the direct product $\prod_{n\in J_i}\alg C_n$ is  continuum. Thus \eqref{eqpbxfMsKtTsl} gives that, for $i\in\set{0,1}$,
\begin{equation}
\text{the cardinality of $\alg A_i$ is continuum; in notation, $|A_i|=2^{\aleph_0}$.} 
\label{eqtxtCntNm}
\end{equation}
Note that the subsequent sections will reference \eqref{eqtxtCntNm} in connection with other structures defined by similar ultraproducts of finite structures. 

The \emph{$\mathbb Z$-chain} is the graph  $\alg C_\infty$ with the set $C_\infty:=\mathbb Z$ of integer numbers as vertex set and $E^{\alg C_\infty}:=\set{\pair x y:  |x-y|=1}$. There is an $\lang\grsigma$-sentence expressing that for every element $x$ there are exactly two elements $y$ such that $xEy$. Apart from $n=2$, this sentence holds in all $\alg C_n$. Hence, by Lo\'s's Theorem again, this sentence holds in $\alg A_0$ and $\alg A_1$. 
This yields that $\alg A_i$ is the disjoint union of copies of $\mathbb Z$-chains and circles $\alg C_k$, $3\leq k\in\mathbb N$.  However, for each $3 \leq k\in\Nplu$, there is an $\lang\grsigma$-sentence expressing that $\alg C_k$ is not a subgraph. This sentence holds in $\alg C_n$ for all $n$ belonging to the cofinite set $J_i\setminus\set{k}$, whereby Lo\'s's Theorem gives that this sentence also holds in $\alg A_i$. Hence, $\alg A_i$ contains no circle. Consequently, for $i\in\set{0,1}$, there is a cardinal number $\kappa_i$ such that $\alg A_i$ is the disjoint union of $\kappa_i$ many copies of $\mathbb Z$-chains.  Combining $|C_\infty|=\aleph_0$ with \eqref{eqtxtCntNm}, it follows that $\kappa_0=2^{\aleph_0}=\kappa_1$. Thus, $\alg A_0$ and $\alg A_1$ are isomorphic graphs; in notation, $\alg A_0\cong \alg A_1$. 

Next, for the sake of contradiction, suppose that Proposition~\ref{propositionbipartite} fails. Then, using that finitely many sentences can always be replaced by their conjunction, there exists a single sentence $\phi$ such that for every finite graph $\alg B$, $\phi$ holds in $\alg B$ if and only if $\alg B$ is a bipartite graph. In particular,
\begin{equation}
\text{$\phi$ holds in $\alg C_n$ for all $n\in J_0$ but it fails in $\alg C_m$ for all $m\in J_1$.}
\label{eqtxtffZfgSr}
\end{equation}
By Lo\'s's Theorem, $\phi$ holds in $\alg A_0$. Hence, by the isomorphism $\alg A_0\cong \alg A_1$,  $\phi$ holds in $\alg A_1$, too. 
Using Lo\'s's Theorem again, we obtain that the set $\{m\in J_1: \phi$ holds in $\alg C_m\}$ belongs to the ultrafilter $U_1$. This contradicts the fact that this set is empty by \eqref{eqtxtffZfgSr}, completing the proof of Proposition~\ref{propositionbipartite}.
\end{proof}

It is worth comparing Proposition~\ref{propositionbipartite} with
 the following folkloric fact. 

\begin{remark}\label{remarkszGtmQ}Let $\sigma$ be a finite signature.  If $\alg K$ is a class of 
 \emph{finite} $\sigma$-structures such that it is closed with respect to taking isomorphic copies, then there exists a
set $\Phi$ of $\lang\sigma$-sentences such that a finite structure belongs to $\alg K$ if and only if it satisfies every member of $\Phi$. 
\end{remark}

\begin{proof}[Proof of Remark~\ref{remarkszGtmQ}]
For each $k\in\Nplu$,  $\alg K$ contains finitely many $k$-element structures (up to isomorphism). Hence there is an $\lang\sigma$-sentence $\nu_k$ that holds exactly in the $k$-element structures of $\alg K$. Thus, we can let 
 $\Phi:=\set{\lambda_k\then\nu_k: k\in\Nplu}$.
\end{proof}

\section{Groups}\label{sectiongroups}
Using the terminology of  Proposition~\ref{propositionbipartite}, we have the following statement.

\begin{proposition}\label{propositiongroups} The class of finite simple groups is not finitely axiomatizable modulo finiteness.
\end{proposition}

\begin{proof}Since lots of arguments used in Proposition~\ref{propositionbipartite} apply here, we give less details.
According to \eqref{eqsignature}, 
the signature $\grsigma$ for groups is chosen so that  $p=0$, $q=1$, $f_1=2$, and $F_1$ is ``+''.
For the sake of contradiction, suppose that there exists an $\lang\grsigma$-sentence $\phi$ that holds in all finite simple groups but it fails in all finite non-simple groups. For $n\in\Nplu$, the cyclic group of order $n$ will be denoted by $\alg C_n$. Let $p_1<p_2<p_3< \dots$ be the list of all prime numbers,  and define $q_j:=p_jp_{j+1}$ for $j\in \Nplu$. Take a nontrivial ultrafilter $U$ over $\Nplu$. Let $\alg A_0$ and $\alg A_1$ be the ultraproducts
$\prod_{n\in \Nplu}\alg C_{p_i}/U$ and $\prod_{n\in \Nplu}\alg C_{q_i}/U$, respectively.  Observe that \eqref{eqtxtCntNm} is still valid; see the sentence right after it.
For $k\in\Nplu$, define the following sentence of $\lang\grsigma$
with $k$ occurrences of $y$: 
\[\eta_k: \qquad \forall x\exists y\bigl((\dots(y+y)+y)+\dots)+y=x).
\]
Basic facts about linear congruences yield that 
the sets $\{n\in\Nplu: \eta_k$  holds in $\alg C_{p_n}\}$ and $\{n\in\Nplu: \eta_k$  holds in $\alg C_{q_n}\}$ are cofinite and so they belong to $U$. Similarly, with
$k+1$ occurrences of $x$ before the first equality sign, if we define
\[\tau_k: \qquad \forall x\Bigl(\bigl((\dots(x+x)+x)+\dots)+x=x\, \then \,\forall y (x+y=y) \Bigr),
\]
then both $\{n\in\Nplu: \tau_k$  holds in $\alg C_{p_n}\}$ and $\{n\in\Nplu: \tau_k$  holds in $\alg C_{q_n}\}$ are cofinite and belong to $U$. Hence, by Lo\'s's Theorem, $\eta_k$ and $\tau_k$ hold in $\alg A_i$ for all $k\in\Nplu$ and $i\in\set{0,1}$. 
Therefore, the abelian groups $\alg A_0$ and $\alg A_1$ are torsion-free (by the sentences $\tau_k$) and divisible (by the $\eta_k$).  Consequently, they are direct sums of copies of the additive group $\tuple{\mathbb Q,+}$ of rational numbers; see, for example, Kurosh~\cite[page 165]{kurosh} or use the straightforward fact that a torsion-free and divisible abelian group can be considered a vector space over the field of rational numbers. \eqref{eqtxtCntNm} implies  that each of $\alg A_0$ and $\alg A_1$ has  $2^{\aleph_0}$-many direct summands. Hence, $\alg A_0\cong \alg A_1$.

For all $n\in\Nplu$,  $\alg C_{p_n}$ is a simple group and so it satisfies $\phi$. Lo\'s's Theorem gives that $\phi$ holds in $\alg A_0$, whereby it holds in $\alg A_1$ since $\alg A_1\cong\alg A_0$.  Using Lo\'s's Theorem again, we obtain that the set $I:=\{n\in\Nplu: \phi$ holds in $\alg C_{q_n}\}$ belongs to the ultrafilter $U$. But none of the groups $\alg C_{q_n}$ is simple, so none of them satisfies $\phi$, whence $I=\emptyset$. The contradiction $\emptyset=I\in U$ completes the proof of Proposition~\ref{propositiongroups}.
\end{proof}

Note that \eqref{eqtxtCntNm} and the structure theorem of torsion-free divisible abelian groups in the proof above were only used to conclude that $\alg A_0\cong \alg A_1$, but this isomorphism was only needed to ensure that $\alg A_0$ and $\alg A_1$ are elementarily equivalent. There is another way to ensure this elementary equivalence that relies neither on  \eqref{eqtxtCntNm}, nor on the above-mentioned structure theorem: one can use the description of elementary equivalence of abelian groups given by Szmielew~\cite{szmielew}. However, the use of this description would require further $\lang\grsigma$-sentences and would make the proof more complicated.

\section{The ordered sets of join-irreducible congruences of slim semimodular lattices}\label{sectionlatt}

\subsection*{Brief introduction to slim semimodular lattices}
We assume that the reader has some basic familiarity with lattices; if not then a few parts of Burris and Sankappanvar~\cite{burrissank} or  Davey and Priestley~\cite{daveypriestley} or Gr\"atzer~\cite{GrGLT} are recommended.

A lattice $\alg L=\tuple{L;\vee, \wedge}$ is \emph{semimodular} if for any $x,y,z\in L$, the covering relation $x\prec y$ implies that $x\vee z\prec y\vee z$ or $x\vee z= y\vee z$. 
The  lattice $\alg L$ is \emph{slim} if it is finite and the (partially) ordered set $\alg J(\alg L)=\tuple{J(\alg L),\leq}$ of its join-irreducible elements is the union of two chains. We know from  Cz\'edli and Schmidt~\cite[Lemma 2.3]{czgschtJH} that for finite semimodular lattices, this definition of slimness is equivalent to the original one, which is due to Gr\"atzer and Knapp~\cite{GrKnI} but not recalled here. We also know from Cz\'edli and Schmidt~\cite[Lemma 2.2]{czgschtJH} that slim lattices are \emph{planar}; 
however, the term ``slim, planar, semimodular lattice'' frequently occurs in the literature since the original concept of slimness did not imply planarity. Here we write ``slim semimodular lattices'' and these lattices are automatically finite and planar. As usual, the set of congruence relations of a lattice $\alg L$ form a lattice, the \emph{congruence lattice} $\Con {\alg L}$ of $\alg L$. The study of congruence lattices of slim semimodular lattices began with  Gr\"atzer and Knapp~\cite{GrKn3}. These congruence lattices  $\Con {\alg L}=\tuple{\Con {\alg L},\leq}$  are distributive. Hence, by the classical structure theorem of finite distributive lattices, see Gr\"atzer~\cite[Theorem II.1.9]{GrGLT} for example, 
these congruence lattices are economically described by simpler and smaller structures: the ordered sets  $\alg J(\Con {\alg L})=\tuple{\Jir{{\Con\lalg L}},\leq}$  of their join-irreducible elements.

Several properties of the ordered sets $\alg J(\Con {\alg L})$ determined by  slim semimodular lattices $\alg L$ have been discovered; they are summarized in Cz\'edli~\cite{czglamps} and Cz\'edli and Gr\"atzer~\cite{czggg3p3c}.
In fact, the attempt to characterize these $\alg J(\Con {\alg L})$ served as an essential motive to deal with slim semimodular lattices. 
For  surveys of 
these lattices, see
the book chapter Cz\'edli and Gr\"atzer~\cite{czgggbch} and 
Section 2 of Cz\'edli and Kurusa~\cite{czgka}.
Here, as an appetizer to this section of the present paper,  we only mention that slim semimodular lattices were used to strengthen the Jordan--H\"older Theorem for groups from the nineteenth century, see Cz\'edli and Schmidt~\cite{czgschtJH} and Gr\"atzer and Nation~\cite{gratzernation}, and they have led to results in geometry, see 
Cz\'edli \cite{czgcharcirc}--\cite{czgcrossing} and  
Cz\'edli and Kurusa~\cite{czgka} together with the survey given in it.
Since  2007, when  G.\ Gr\"atzer and E.\ Knapp\cite{GrKnI}
introduced slim semimodular lattices, the study of these lattices has been the most intensive part of lattice theory. Indeed, at the time of writing, the MathSciNet search ``Anywhere=(slim and semimodular) AND pubyear in [2012 2021]'' returns 22 matches.

\subsection*{The result of this section and its proof}
In harmony with \eqref{eqsignature}, we assume that ordered sets are of type $\ordsigma=\bigl\langle 1,0, \tuple{\leq,2}\bigr\rangle$. Using this notation and \eqref{eqtxtFslNg}, we formulate our main result as follows.


\begin{theorem}\label{thmnonax}
The class of ordered sets of join-irreducible congruences of slim semimodular lattices is not finitely axiomatizable modulo finiteness. That is, there exists no finite set $\Phi$ of\,  $\lang\ordsigma$-sentences such that a finite ordered set $\alg S=\tuple{S,\leq}$ is isomorphic to the ordered set $\alg J(\Con{\alg L})=\tuple{J(\Con{\alg L}),\leq}$ of some slim semimodular lattice $\alg L$ if and only if all members of\, $\Phi$ hold in $\alg S$.
\end{theorem}

\begin{proof} Suppose the contrary. Then, as in the proof of Proposition~\ref{propositionbipartite}, we can pick a single 
$\lang\ordsigma$-sentence $\phi$ such that a finite ordered set $\alg S$ satisfies $\phi$ if and only if $\alg S\cong \alg J(\Con{\alg L})$ for a slim semimodular lattice $\alg L$. For $2\leq n\in\Nplu$, the \emph{$n$-crown} $\alg K_n$  is the $2n$-element ordered set with maximal elements $a_0,a_1,\dots, a_{n-1}$ and minimal elements $b_0,b_1,\dots, b_{n-1}$ such that, for $i,j\in\set{0,1,\dots, n-1}$,  $b_i\leq a_j$ if and only if $i=j$ or $i+1\equiv j$ (mod $n$). For $n=8$,  $\alg K_n$ is drawn in Figure~\ref{figK8}.

\begin{figure}[ht]
\centerline
{\includegraphics[scale=1]{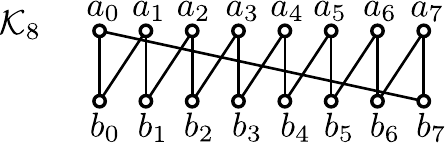}}
\caption{$K_8$}\label{figK8}
\end{figure}

\noindent
We let $J_0=\set{2,4,6,\dots}$ and $J_1=\set{3,5,7,\dots}$. 
Take a nontrivial ultrafilter $U_i$ over $J_i$. For $i\in\set{0,1}$, let $\alg A_i$ be the ultraproduct $\prod_{n\in J_i} \alg K_i/U_i$. 
We know from \eqref{eqtxtCntNm} and the sentence following it that $|A_0|=|A_1|=2^{\aleph_0}$. 
Although, to save space, we do not give all of them in details, we have the following $\lang\ordsigma$-formulas.
\begin{enumerate}
\item[$\alpha(x)$:] $\forall y(x\leq y \then y\leq x )$, which expresses that $x$ is a maximal element.
\item[$\beta(x)$:] $\forall y(y\leq x \then x\leq y )$, which expresses that $x$ is a minimal element.
\item[$\delta_1$:]  $\forall x$, exactly one of $\alpha(x)$ and $\beta(x)$ holds. 
\item[$\delta_2$:]  $\forall x$, if $\alpha(x)$, then there are exactly two elements $y$ such that $\beta(y)$ and $y\leq x$. 
\item[$\delta_3$:]  $\forall x$, if $\beta(x)$, then there are exactly two elements $y$ such that $\alpha(y)$ and $x\leq y$.
\item[$\xi_m$:] there are no elements forming a subset order isomorphic to $\alg K_m$.
\end{enumerate}
The ordered set
$\alg F=\tuple{\set{a_j: j\in\mathbb Z}\cup \set{b_j: j\in\mathbb Z},\leq}$  such that $\alpha(a_j)$, $\neg \beta(a_j)$, $\beta(b_j)$, and $\neg \alpha(a_j)$ for all $j\in \mathbb Z$ and, in addition, 
$b_j\leq a_s$ if and only if $s\in\set{j,j+1}$ will be called an (infinite) \emph{fence}; see Figure~\ref{figfence}.

\begin{figure}[ht]
\centerline
{\includegraphics[scale=1]{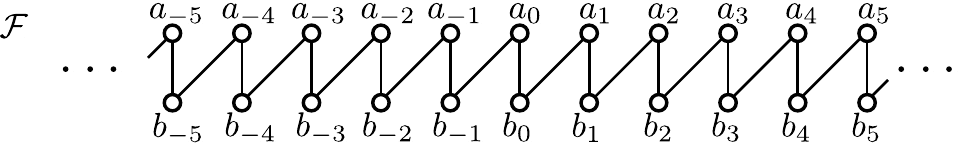}}
\caption{$\alg F$}\label{figfence}
\end{figure}

Recall that for ordered sets $\alg W_h=\tuple{W_h,\leq^h}$, $h\in H$, we obtain the \emph{cardinal sum} $\alg W=\tuple{W,\leq}$  of these ordered sets  by letting $W$ be the disjoint union of the $W_h$, $h\in H$, and defining  $\leq$ as the union of the $\leq^h$, $h\in H$.
Since $\delta_1$, $\delta_2$, and $\delta_3$ hold in  $\alg K_n$ for all $n\in J_0\cup J_1$ and, for each $m\geq 2$, so does $\xi_m$  for all $n\in (J_0\cup J_1)\setminus \set{m}$,   Lo\'s's Theorem yields that $\delta_1$, $\delta_2$, $\delta_3$, and, for all $m\in \Nplu\setminus\set 1$, $\xi_m$ hold in $\alg A_0$ and $\alg A_1$. Therefore, for $i\in\set{0,1}$, we conclude that 
each element of $\alg A_i$ belongs  to a unique fence and $\alg A_i$ is the cardinal sum of some copies, say $\kappa_i$ copies, of fences.  Using $|A_0|=2^{\aleph_0}=|A_1|$, we obtain that $\kappa_0=2^{\aleph_0}=\kappa_1$. Therefore, $\alg A_0\cong \alg A_1$.

\begin{figure}[ht]
\centerline
{\includegraphics[scale=1]{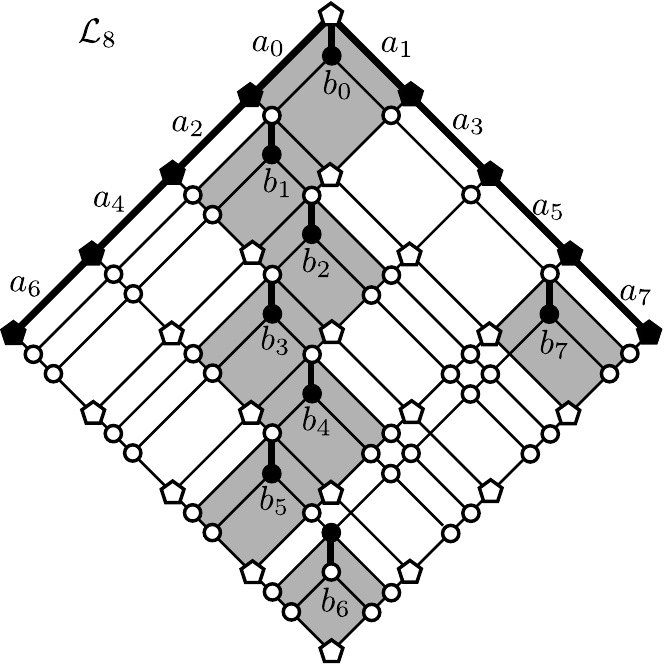}}
\caption{$\alg L_8$}\label{figL8}
\end{figure}

The rest of the proof relies heavily on Cz\'edli~\cite{czgtrajcolor} and mainly on \cite{czglamps}; these two papers\footnote{\red{Temporary note: see 
\texttt{ http://www.math.u-szeged.hu/\textasciitilde{}czedli/ } for their preprints.}} should be near. 
In particular, the  notation and the concepts not defined here are given there. However, we recall the following definition from Cz\'edli~\cite{czglamps}.

\begin{definition}[Bipartite Maximal Elements Property]\label{def:bzwBNrtmG}
A finite distributive lattice $\lalg D$ satisfies the \emph{Bipartite Maximal Elements Property} if $\maxJir {\lalg D}$ can be represented as the disjoint union of two nonempty subsets such that no two distinct elements in the same subset have a common lower cover (with respect to $\prec_{\aJir {\lalg D}}$) in the poset $\aJir {\lalg D}=\tuple{\Jir{\lalg D},\leq}$.
\end{definition}

We know from Cz\'edli~\cite[Corollary 3.4 ]{czglamps} that the congruence lattice of a slim semimodular lattice satisfies the Bipartite Maximal Elements Property. 

Resuming the proof of Theorem~\ref{thmnonax}, 
observe that if $\aJir {\lalg D}\cong \alg K_n$ for some $n\in J_1$, then ${\lalg D}$ fails to satisfy the Bipartite Maximal Elements Property. Therefore we obtain from \cite[Corollary 3.4 ]{czglamps} that  $\alg K_n\cong \alg J(\Con{\alg L})$ with a slim semimodular $\alg L$ cannot hold if $n\in J_1$. For later reference (at Remark~\ref{rem:nWhWrlNsW}), note at this point that in addition to proving Theorem~\ref{thmnonax}, our argument will automatically yield that 
\begin{equation}
\parbox{8.2cm}{there is no finite set $\Phi$ of first-order formulas in $\lang\ordsigma$ such that, for any finite distributive lattice ${\lalg D}$, $\Phi$ holds in $\aJir {\lalg D}$ if and only if ${\lalg D}$ satisfies the Bipartite Maximal Elements Property.}
\label{eqpbx:swGjwmknHpRc}
\end{equation}

We have just pointed out that $\phi$ implies the Bipartite Maximal Elements Property.
Now it follows that $\phi$ fails in $\alg K_n$ for $n\in J_1$, and Lo\'s's Theorem gives that $\phi$ does not hold in $\alg A_1$. 

Next, we assume that $n\in J_0$. Let $k:=n/2$. To construct a lattice, 
we begin with the direct square of the $(k+1)$-element chain; it is a distributive lattice called a \emph{grid}. For $n=8$, this grid consists of the pentagon-shaped elements in Figure~\ref{figL8}. Going downwards, we label the edges on the upper left boundary by $a_0$, $a_2$, \dots, $a_{n-2}$.
Also, we label the edges on the upper right boundary by $a_1$, $a_3$, \dots, $a_{n-1}$,  going downwards again. In this way, we have labeled the non-vertical thick edges of Figure~\ref{figL8}. 
At this stage, the circle-shaped elements and the edges having (at least one) circle-shaped endpoints are not present. 
The edges of the grid determine $k^2$ many 4-cells (that is, squares) in the plane. 
We obtain a slim semimodular lattice $\alg L_n$ from the grid in $n$ steps in the following way. 
First, we insert a fork (that is, a multifork of rank 1) into
$\REnl{a_0}\cap \LEnl{a_1}$; this intersection is the uppermost grey-filled rectangle (which happens to be a square) in the figure. This insertion brings the $b_0$-labeled thick vertical edge in. 
(Since there would not be enough room otherwise, the label of a vertical thick edge is always below the edge in Figure~\ref{figL8}; note that the thick edges are exactly the labeled edges.)
In the second step, we insert a fork  into
$\REnl{a_0}\cap \LEnl{a_{n-1}}$, understood in the lattice obtained in the previous step, of course. In the figure, this step brings the $b_7$-labeled thick vertical edge in, and the intersection in question as well as the subsequent intersections are grey-filled.  In the third step, we insert a fork into $\LEnl{a_1}\cap \REnl{a_2}$
and we obtain the $b_1$-labeled thick vertical edge. 
And so on, inserting a fork into $\REnl{a_2}\cap \LEnl{a_3}$, $\LEnl{a_3}\cap \REnl{a_4}$, $\REnl{a_4}\cap \LEnl{a_5}$, \dots, 
$\REnl{a_{n-2}}\cap \LEnl{a_{n-1}}$, one by one and in this order, we obtain the thick vertical edges with labels  $b_2$, $b_3$, $b_4$, \dots, $b_{n-2}$, respectively. 
After performing these steps, we obtain the required lattice  $\alg L_n$. For $n=8$, $\alg L_n= \alg L_8$  is given in Figure~\ref{figL8}. By Theorem 3.7 of Cz\'edli~\cite{czgtrajcolor},  $\alg L_n$ is a slim semimodular lattice. By (the Main) Lemma 2.11 of Cz\'edli~\cite{czglamps}, $\alg K_n\cong \alg J(\Con{\alg L_n})$. This isomorphism and the choice of $\phi$ gives that $\phi$ holds in $\alg K_n$. This is true for all $n\in J_0$, whereby  Lo\'s's Theorem implies that $\phi$ holds in $\alg A_0$.   But this is a contradiction 
since $\alg A_0\cong \alg A_1$ but we have previously seen that $\phi$ does not hold in $\alg A_1$. The proof of Theorem~\ref{thmnonax} is complete.
\end{proof}

\section{A new property and the limits of our method}\label{sect:morelatt}

In the section, we present a new property of the congruence lattices of slim semimodular lattices. Using this property, we point out that the construction given in the previous section is not appropriate to strengthen Theorem~\ref{thmnonax} from finite posets and $\aJir{\Con {\lalg L}}$ to finite distributive lattices and $\Con {\lalg L}$. The concepts and notations given in Section~\ref{sectionlatt} remain in effect but  we also need some additional definitions and notations.

\begin{definition}[Two-cover Property by Gr\"atzer~\cite{gr:confork}]\label{def:pbxwzwWrr} A finite distributive lattice ${\lalg D}$ satisfies the \emph{Two-cover Property} if every element of the poset $\aJir {\lalg D}$ has at most two covers in $\aJir {\lalg D}$ (with respect to $\prec_{\aJir {\lalg D}}$).
\end{definition}

Recall the following result from Gr\"atzer~\cite{gr:confork}.

\begin{lemma}[{Gr\"atzer~\cite[Theorem 1.5]{gr:confork}}]\label{lemma:twocoverP}
The congruence lattice $\Con {\lalg L}$ of a  slim semimodular lattice ${\lalg L}$ satisfies the Two-cover Property. 
\end{lemma}

While the Two-cover Property for a finite distributive lattice ${\lalg D}$ can easily be described by a first-order formula both in $\aJir {\lalg D}$ and in ${\lalg D}$, \eqref{eqpbx:swGjwmknHpRc} indicates the situation for the Bipartite Maximal Elements Property is more involved.

For an element $y$ of a lattice ${\lalg L}=\tuple{L,\leq}$, the \emph{principal ideal} $\set{x\in L: x\leq y}$ will be denoted by $\ideal y$. For a finite distributive lattice ${\lalg D}$, the \emph{set of the maximal elements} of the poset $\aJir {\lalg D}$ will be denoted by $\maxJir {\lalg D}$. 
One of the key concepts in this section is given in the following definition. Motivated by Lemma~\ref{lemma:twocoverP}, we do not define this concept for ${\lalg D}$ \emph{not satisfying} the Two-cover Property.

\begin{definition}[Cyclic element]\label{def:cyclelem} 
Let ${\lalg D}$ be a finite distributive lattice satisfying the Two-cover Property; see Definition~\ref{def:pbxwzwWrr}. 
An element $x$ of  ${\lalg D}$ is  a \emph{cyclic element} if there is an integer $n\geq 3$ and an $n$-tuple $\tuple{a_0,a_1,\dots, a_{n-1}}$ of pairwise distinct elements of $\maxJir {\lalg D}$ such that 
\begin{align}
x&= a_0\vee a_1\vee\dots\vee a_{n-1},\label{align:swzwtGa}
\end{align}
and, for all $0\leq i<j\leq n-1$, 
\begin{equation}
\Jir {\lalg D}\cap \ideal{a_i}\cap\ideal{a_j}\text{ is nonempty if and only if }j=i+1 \text{ or } \tuple{i,j}=\tuple{0,n-1}.
\label{eq:cycLwzT}
\end{equation}
\end{definition}

\begin{figure}[ht]
\centerline
{\includegraphics[scale=1.1]{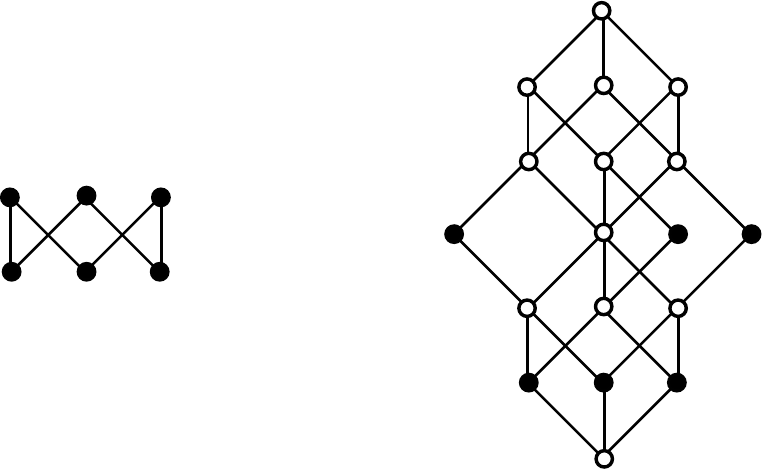}}
\caption{$\FDthree$ on the right and $\aJir{\FDthree}$  on the left}\label{figFD3}
\end{figure}
The free distributive lattice $\FDthree$ on three generators is shown on the right of Figure~\ref{figFD3} while $\aJir{\FDthree}$ is on the left of the figure.  (This figure and Proof~\ref{rem:skVslGjR} below are only included in this extended version of the paper.) Observe that $\aJir{\FDthree}$ satisfies the Two-cover property.

\begin{remark}\label{rem:skVslGjR}
The top element of the free distributive lattice $\FDthree$ on three generators is a cyclic element. 
Also, if ${\lalg D}$ is a distributive lattice with less than $|\FDthree|=18$ elements such that ${\lalg D}$ satisfies the Two-cover Property, see Definition~\ref{def:pbxwzwWrr}, then ${\lalg D}$ has no cyclic element.
\end{remark}

\begin{proof}[Proof of Remark \ref{rem:skVslGjR}] The first sentence of the remark needs no proof. 

 For a poset $\lalg P=\tuple{P,\leq}$, a subset (possibly empty subset) $X$ of $P$ is a \emph{down-set} if for every $u\in X$ and $v\in P$, $v\leq u$ implies that $v\in X$. With respect to set inclusion ``$\subseteq$'', the \emph{set $\Id {\lalg P}$ of all down-sets} of ${\lalg P}$ is a distributive lattice. By the  structure theorem of finite distributive lattices, see Gr\"atzer~\cite[Theorem II.1.9]{GrGLT} for example, every finite distributive lattice ${\lalg E}$ is (isomorphic to)  $\Id{\aJir{\lalg E}}$. The  $n$-crown $\alg K_n=\tuple{K_n,\leq}$ has been defined in the proof of Theorem~\ref{thmnonax}; see Figures~\ref{figK8} for $\alg K_8$ and the left of Figure~\ref{figFD3} for $\alg K_3$. 

Assume that ${\lalg D}$ is a distributive lattice having a cyclic element and satisfying the Two-cover Property; we need to show that ${\lalg D}$ has at least 18 elements.  Due to our assumption, there exists an $n$ such that $3\leq n\in\Nplu$ and  $\alg K_n$ is a subposet of $\aJir {\lalg D}$. For a down-set $X$ of $\aJir {\lalg D}$, 
$X\cap  K_n$ is a down-set of $\alg K_n$. So we have a map $\phi\colon \Id{\aJir {\lalg D}}\to \Id{\alg K_n}$, defined by $\phi(X):=X\cap  K_n$. This map is surjective since for each $Y\in\Id{\alg K_n}$, 
$X:=\{x\in \Jir {\lalg D}: x\leq y$ for some $y\in Y\}$ is a pre-image of $Y$, that is, $\phi(X)=Y$. Using the surjectivity of $\phi$ and the structure theorem mentioned above, we have that 
$|D|=|\Id{\aJir {\lalg D}}|\geq |\Id{\alg K_n}|$. 
Now if $n=3$, then Figure~\ref{figFD3} shows that $\Id{\alg K_n}=\Id{\alg K_3}=\FDthree$, which has 18 elements. If $n\geq 4$, then $\Id{\alg K_n}$ has even more elements since it has  $2^n\geq 16$ elements not containing any maximal element of $\alg K_n$ and at least $2^n-1\geq 15$ elements that contain at least one maximal element of $\alg K_n$.   
\end{proof}



\begin{definition}[V-set and W-set]\label{def:MsetMMset} 
Let ${\lalg D}$ be a finite distributive lattice and let $x\in {\lalg D}$. 
A \emph{V-set} of $x$ is a two-element subset $\set{a_0,a_1}$ of $\maxJir {\lalg D}\cap\ideal x$ such that 
$\Jir {\lalg D}\cap \ideal{a_0}\cap\ideal{a_1}$ is nonempty. Similarly, a \emph{W-set} of $x$ is a four-element subset
$Y$ of $\maxJir {\lalg D}\cap\ideal x$  that can be written in the form $Y=\set{a_0,a_1,a_2,a_3}$ such that
\begin{equation}
\text{for $0\leq i<j\leq 3$, 
$\Jir {\lalg D}\cap \ideal{a_i}\cap\ideal{a_j}\neq\emptyset$ if and only if $j = i+1$.}
\label{eq:MndMMsTrhG}
\end{equation}
\end{definition}

The terminology ``V-set'' is explained by the diagram of the three-element poset consisting of $a_0$, $a_1$, and a lower bound of $a_0$ and $a_1$. (According to this visualization,  a ``triple V'' character would be better than ``W''  in the name ``W-set'', but there is no such character.)

\begin{definition}[VW-element]\label{def:MMMelmnt} 
An element $x$ of a  finite distributive lattice ${\lalg D}$ is
a \emph{VW-element} if 
\begin{equation}
0\neq x=\bigvee \Bigl(\maxJir {\lalg D}\cap \ideal x\Bigl)
\label{eq:WrjRslGzT}
\end{equation}
and, in addition, for each $y\in \maxJir {\lalg D}\cap \ideal x$, exactly one of the following two conditions holds:
\begin{align}
&\text{either $y$ belongs to a unique V-set of $x$,}
\label{align:slTmrBma}
\\
&\text{or $y$ belongs to a unique W-set of $x$.}
\label{align:slTmrBmb}
\end{align}
\end{definition}

Armed with Definitions~\ref{def:MsetMMset}--\ref{def:MMMelmnt}, we define the following property.

\begin{definition}[Decomposable Cyclic Elements Property]\label{def:decCycEls} 
A finite distributive lattice ${\lalg D}$ satisfies the \emph{Decomposable Cyclic Elements Property} if
\begin{equation} 
\parbox{6.6cm}{for each cyclic element $x\in {\lalg D}$, there exist VW-elements $y,z\in {\lalg D}$ such that $x=y\vee z$ and $\maxJir {\lalg D}\cap\ideal y \cap \ideal z=\emptyset$.}
\label{eqpbx:nWpr}
\end{equation}
\end{definition}

Based on Remark~\ref{rem:skVslGjR}, it is easy to see the following.

\begin{remark}
The free distributive lattice $\FDthree$ fails to satisfy the Decomposable Cyclic Elements Property. 
For every distributive lattice ${\lalg D}$ with less than 18 elements, if ${\lalg D}$ satisfies the Two-cover Property (see Definition~\ref{def:pbxwzwWrr}), then ${\lalg D}$ satisfies the Decomposable Cyclic Elements Property. 
\end{remark}

In harmony with \eqref{eqsignature}, we assume that lattices are of type $\ordsigma=\bigl\langle 1,0, \tuple{\leq,2}\bigr\rangle$. However, since the equalities $x\vee y=z$ and $x\wedge y=t$ are easily expressible in the first-order language $\lang\ordsigma$, we can use the binary operation symbols $\vee$ and $\wedge$ without loss of generality. By a similar reason, we can also use the binary relation symbols $<$ and $\neq$.
Below, we formulate the main result of this section; for the concepts occurring in it, see Definitions~\ref{def:bzwBNrtmG}, \ref{def:pbxwzwWrr}, and \ref{def:decCycEls}.

\begin{theorem}\label{thm:newproptY} \ 

\textup{(A)} If ${\lalg L}$ is a slim semimodular lattice, then its congruence lattice satisfies the Decomposable Cyclic Elements Property.  

\textup{(B)} Let ${\lalg D}$ be a finite distributive lattice satisfying the Two-cover Property. Then ${\lalg D}$ satisfies the Decomposable Cyclic Elements Property if and only if it satisfies the Bipartite Maximal Elements property. 

\textup{(C)} The Decomposable Cyclic Elements Property is equivalent to a first-order formula of $\lang\ordsigma$ in \emph{finite} distributive lattices satisfying the Two-cover Property.
\end{theorem}

\begin{proof} 
We will use the following well-known property of distributive lattices ${\lalg D}$: 
\begin{equation}
\parbox{8.1cm}{if $x\in \Jir {\lalg D}$, $y_1,\dots,y_t\in D$, and $x\leq y_1\vee \dots \vee y_t$, then there is an $i\in\set{1,\dots,t}$ such that $x\leq y_i$.}
\label{eqpbx:wsFjERnQvCz}
\end{equation}
Indeed, the premise of \eqref{eqpbx:wsFjERnQvCz} and distributivity yield that $x=(x\wedge y_1)\vee\dots\vee (x\wedge y_t)$, whereby $x\in\Jir {\lalg D}$ implies easily that $x\leq y_i$ for some $i$. Using \eqref{eqpbx:wsFjERnQvCz} and the fact that $\maxJir {\lalg D}$ is an antichain, we conclude that, for arbitrary $t\in \Nplu$ and $a,b_1,\dots b_t$,
if $a,b_1,\dots,b_t\in\maxJir {\lalg D}$ and $a\leq b_1\vee\dots\vee b_t$, there is an $i\in\set{1,\dots,t}$ such that $a=b_i$. Hence, for every $Y$, 
\begin{equation}
\text{if $Y\subseteq \maxJir {\lalg D}$ and $y:=\bigvee Y$, then
$Y=\maxJir {\lalg D}\cap\ideal y$.
}
\label{pbx:sZrkjNhZzafBrnsn}
\end{equation}
For $x,y\in \maxJir {\lalg D}$, we let
\begin{equation}
x\relE y \defiff \Bigl( x\neq y\text{ and }\Jir {\lalg D}\cap
\ideal x\cap \ideal y\neq \emptyset \Bigr).
\label{eq:relEowg}
\end{equation} 
Based on \eqref{eq:relEowg} and Definition~\ref{def:grSnTr},
$\tuple{\maxJir {\lalg D}, E }$ is a graph. Note that by a \emph{graph} we mean an undirected graph without loop edges and multiple edges. Clearly,
\begin{equation}
\parbox{8cm}{${\lalg D}$ satisfies the Bipartite Maximal Elements Property if and only if $\tuple{\maxJir {\lalg D}, E }$ is a bipartite graph.}
\label{eqpbx:LssVtdJGrhtzRj}
\end{equation}
For graphs $\tuple{G_1,\relE_1}$ and $\tuple{G_2,\relE_2}$, we say that  $\tuple{G_1,\relE_1}$ is a \emph{spanned subgraph} of 
$\tuple{G_2,\relE_2}$ if $G_1\subseteq G_2$ and $\relE_1$ is $\relE_2\cap (G_1\times G_1)$, that is, $\relE_1$ is the restriction of $\relE_2$ to $G_1$. By a \emph{spanned circle} in a graph we mean a spanned subgraph that is a circle of length at least 3.

First, we deal with part (B). For the sake of contradiction, suppose that ${\lalg D}$  satisfies the Decomposable Cyclic Elements Property \eqref{eqpbx:nWpr} but fails to satisfy the
Bipartite Maximal Elements Property. Then the graph $\tuple{\maxJir {\lalg D}, E }$ is not bipartite by \eqref{eqpbx:LssVtdJGrhtzRj}. Therefore, as it is well known, this graph contains a circle of an odd length. 
By finiteness, we can pick a circle $\tuple{a_0,a_1,\dots a_{n-1}}$ of a \emph{minimal odd length} $n$.  By \eqref{eq:relEowg}, the graph contains no loop edge, whence $n\geq 3$. The notation $\tuple{a_0,a_1,\dots a_{n-1}}$ is understood so that   
$a_{0}\relE a_{1}$, $a_{1}\relE a_{2}$, \dots, $a_{n-2}\relE a_{n-1}$, and $a_{n-1}\relE a_{0}$. 
We claim that 
\begin{equation}
\parbox{8.3cm}{$\tuple{a_0,a_1,\dots a_{n-1}}$ is a \emph{spanned} circle in $\tuple{\maxJir {\lalg D}, E }$, that is, for $i,j\in\set{0,1,\dots,n-1}$, $a_i\relE a_j$ if and only if  $|j-i|\in\set{1, n-1}$.}
\label{eqtxtazTg}
\end{equation}
The ``if'' part is clear since  $\tuple{a_0,a_1,\dots a_{n-1}}$ is a circle. For the sake of contradiction, suppose that $a_i\relE a_j$ such that $|a_i-a_j|\notin \set{1,n-1}$. Since $i\neq j$ and $\relE$ is symmetric, we can assume that $j=i+r$ for some $r\in\set{2,3,\dots,n-2}$. Then 
$\tuple{a_i,a_{i+1},\dots, a_j}$ is a circle of length $r+1$ while $\tuple{a_j,\dots, a_{n-1},a_0,\dots, a_i}$ is a circle of length $n-r+1$. Both circles have lengths smaller than $n$. Since $(r+1)+(n-r+1)=n+2$ is odd, one of these two circles is of an odd length. This contradicts the minimality of $n$ and proves \eqref{eqtxtazTg}.

Next, using the spanned circle  in \eqref{eqtxtazTg}, we define $x$ by $x:=a_0\vee a_1\vee\dots \vee a_{n-1}$. Clearly, \eqref{align:swzwtGa} holds. Since \eqref{eqtxtazTg} implies \eqref{eq:cycLwzT}, $x$ is a cyclic element of ${\lalg D}$. Since we have assumed that ${\lalg D}$ satisfies the Decomposable Cyclic Elements Property, we can pick VW-elements $y$ and $z$ according to \eqref{eqpbx:nWpr}. Let $I_y:=\set{i: 0\leq i\leq n-1\text{ and }a_i\leq y}$ and $I_z:=\set{i: 0\leq i\leq n-1\text{ and }a_i\leq z}$. By $x=y\vee z$ and \eqref{pbx:sZrkjNhZzafBrnsn}, we have that  $\set{a_0,\dots, a_{n-1}}=I_y\cup I_z$. Since $\ideal y\cap \ideal z$ is disjoint from $\maxJir {\lalg D}$ by \eqref{eqpbx:nWpr}, $I_y\cap I_z=\emptyset$. Thus, $n=|I_y|+|I_z|$. By the definition of a VW-element, $\set{a_i:i\in I_y}$ is a disjoint union of V-sets and W-sets of $y$. (We allow the case when only V-sets occur as well as the case when only W-sets occur.) A V-set and a W-set consist of two and four elements, respectively, whereby $|\set{a_i:i\in I_y}|=|I_y|$ is an even number. We obtain similarly that $|I_z|$ is even. 
Hence, $n=|I_y|+|I_z|$ is an even number. This contradicts the choice of $n$ and proves the ``only if'' implication of part (B).

Second, assume that ${\lalg D}$  satisfies the Bipartite Maximal Elements Property. Let $x\in D$ be a cyclic element. By 
Definition~\ref{def:cyclelem}, there is an $n\geq 3$ and an 
$n$-tuple $\tuple{a_0,a_1,\dots, a_{n-1}}$ of pairwise distinct elements of $\maxJir {\lalg D}$ that  satisfies \eqref{align:swzwtGa} and \eqref{eq:cycLwzT}. Since we obtain from   \eqref{eq:cycLwzT} that $\tuple{a_0,a_1,\dots, a_{n-1}}$ is a spanned circle in the graph $\tuple{\maxJir {\lalg D},\relE}$ and this graph is bipartite, we obtain that $n$ is an even number. There are two cases depending on whether $n$ is divisible by 4 or not.
If $n$ is divisible by 4, then define
\begin{equation}
\begin{aligned}
U&=\set{a_0,a_1,\,\, a_4,a_5,\,\, a_8,a_9,\,\,\dots,\,\, a_{n-4}, a_{n-3} },\cr
V&=\set{a_2,a_3,\,\, a_6,a_7,\,\, a_{10},a_{11},\,\,\dots,\,\, a_{n-2}, a_{n-1}},
\end{aligned}
\label{eq:sjfprsngYl}
\end{equation}
and let $y:=\bigvee U$ and $z:=\bigvee V$. 
By \eqref{pbx:sZrkjNhZzafBrnsn}, $U=\maxJir {\lalg D}\cap\ideal y$. 
It is clear by \eqref{eq:cycLwzT} that  $\set{a_0,a_1}$, $\set{a_4,a_5}$, $\set{a_8,a_9}$,  \dots, $ \set{a_{n-4}, a_{n-3}}$ is the complete list of V-sets of $y$ and $y$ has no W-set.  Hence, we conclude that $y$ a VW-element. We obtain similarly that so is $z$. 
Using $U=\maxJir {\lalg D}\cap\ideal y$, $V=\maxJir {\lalg D}\cap\ideal z$ and $U\cap V=\emptyset$, we obtain that $\maxJir {\lalg D}\cap\ideal y\cap \ideal z=\emptyset$. Since $U\cup V=\set{a_0,a_1,\dots, a_{n-1}}$, \eqref{align:swzwtGa} gives that $y\vee z=x$. Thus, the requirements of \eqref{eqpbx:nWpr} hold for $x$ provided $n$ is divisible by 4.

Next, assume that $n$ is not divisible by 4 and keep it in mind that $n$ is even. Since $n\geq 3$, we have that $n\geq 6$.  Let
\begin{equation}
\begin{aligned}
U&=\set{a_0,a_1,
a_2,a_3,\,\, a_6,a_7,\,\, a_{10},a_{11},\,\,\dots,\,\, a_{n-4}, a_{n-3}},\cr
V&=\set{a_4,a_5,\,\, a_8,a_9,\,\, a_{12},a_{13},\,\,\dots,\,\, a_{n-2}, a_{n-1}},
\end{aligned}
\label{eq:sjfprsngYl}
\end{equation}
Define $y:=\bigvee U$ and $z:=\bigvee V$ as before. 
It follows from \eqref{eq:cycLwzT} that $y$ has exactly one W-set, $\set{a_0,a_1,a_2,a_3}$ and the V-sets of $y$ are 
$\set{a_6,a_7}$, $\set{a_{10},a_{11}}$, \dots,  $\set{a_{n-4}, a_{n-3}}$. Hence, $y$ is a VW-element. Similarly to the case $4\mid n$, we have that $z$ is also a VW-element.
The equalities $\maxJir {\lalg D}\cap\ideal y\cap \ideal z=\emptyset$ and $y\vee z=x$ follow in the same way as in the $4\mid n$ case. 
Thus, the requirements of \eqref{eqpbx:nWpr} hold again, proving  the ``if'' direction  of part (B).
We have proved part (B) of the Theorem~\ref{thm:newproptY}.

By Corollary 3.4 in \cite{czglamps}, the congruence lattice $\Con {\lalg L}$ of a slim semimodular lattice ${\lalg L}$ satisfies the Bipartite Maximal Elements Property. This fact and part (B) imply part (A) of  the Theorem~\ref{thm:newproptY}.

Next, we turn our attention to part (C). 
Since we do not know how to distinguish between a large spanned circle and the union of several large spanned circles in the language $\lang\ordsigma$, we cannot base the axiomatization of the Cyclic Decomposable Elements Property merely on the concept of spanned circles. We define a new concept as follows. 
By  a \emph{spanned multicircle} of a graph $\tuple{G_2,\relE_2}$ we mean a spanned subgraph whose (connectivity) components are spanned circles in $\tuple{G_2,\relE_2}$. For example, if $\set{a_0,a_1,a_2, b_0,b_1, b_2, b_3}$ is a seven element subset of $G_2$ and, apart from symmetry, $a_0 \relE_2 a_1$, $a_1\relE_2 a_2$, $a_2\relE_2 a_0$ and  $b_0 \relE_2 b_1$, $b_1\relE_2 b_2$, $b_2\relE_2 b_3$, $b_3 \relE_2 b_0$ are the only edges among these seven elements, then these seven elements form a spanned multicircle in $\tuple{G_2,\relE_2}$. An element $x$ of a finite distributive lattice ${\lalg D}$ will be called a \emph{multicyclic element} if  there is a multicircle $H$ in the graph $\tuple{\maxJir {\lalg D}\cap \ideal x,\relE}$ such that $x=\bigvee H$.  Clearly, every cyclic element is multicyclic but not conversely. 
Note that $\tuple{\maxJir {\lalg D}\cap \ideal x,\relE}$ is a spanned subgraph of $\tuple{\maxJir {\lalg D},\relE}$. Hence, a spanned multicircle in  $\tuple{\maxJir {\lalg D}\cap\ideal x,\relE}$ (like $H$ above in the definition of a multicyclic element) is also a spanned multicircle in $\tuple{\maxJir {\lalg D},\relE}$.

Next, we define the following relational symbols; we are going to use them for a finite distributive lattice ${\lalg D}$; note that we do not make a notational distinction between these symbols and the relations they define on ${\lalg D}$.
\begin{align}
&\text{$\rjir y x$ means that $y\in \Jir {\lalg D}$ and $y\leq x$.} 
\label{align:rSmbLa}
\\
&\text{$\rmjir y x$ means that $y\in \maxJir {\lalg D}$ and $y\leq x$.}
\label{align:rSmbLb}
\\
&\text{$\rmset {y_0}{y_1}x$ means that $\set{y_0,y_1}$ is a V-set of $x$.}
\label{align:rSmbLc}
\\
&\text{$\rmmset {y_0}{y_1}{y_2}{y_3}x$ means that $\set{y_0,y_1,y_2,y_3}$ is a W-set of $x$.}
\label{align:rSmbLd}
\\
&\text{$\rmmm y$ means that $y$ is a VW-element.}
\label{align:rSmbLe}
\end{align}
On ${\lalg D}$, as it is easy to see, each of these five relations can be described by a first-order formula belonging to
$\lang\ordsigma$. Let $\orjir y x$, \dots,  $\ormmm y$ be first-order formulas in $\lang\ordsigma$ describing the relations defined in \eqref{align:rSmbLa}, \dots, \eqref{align:rSmbLe}, respectively. Let
\begin{equation}
\text{$\rmcyclic x$ mean that $x$ is a multicyclic element.}
\label{align:rSmbLf}
\end{equation}
We are going to show that $\rmcyclic x$ is equivalent to a first-order formula $\ormcyclic x$ belonging to $\lang\ordsigma$. 
To do so,  first we define some smaller formulas; each of them will be followed by an explanation in the text.
\begin{equation*}
\parbox{9cm}{
$\rho^\ast_{\textup{edge}}(y_1,y_2,x):=  \ormjir {y_1} x \mathand
\ormjir {y_2} x  \mathand y_1\neq y_2 
\\
\phantom{mmmmmmmmm}\mathand (\exists z)\bigl(z< y_1\mathand z<y_2\mathand \orjir z x
\bigr);$
} 
\end{equation*}
here and later ``$\mathand$'' means conjunction. The meaning of $\rho^\ast_{\textup{edge}}(y_1,y_2,x)$ is that $y_1$ and $y_2$ are the endpoints of an edge in the graph $\tuple{\maxJir {\lalg D} \cap \ideal x, \relE}$ and, in particular, they are distinct elements of the principal ideal $\ideal x$. Let
\begin{equation*}
\parbox{9cm}{
$\rho^\ast_{\exists2}(x):= (\forall y)\Bigl(\ormjir y x\then 
(\exists y_1)(\exists y_2)\bigl( \rho^\ast_{\textup{edge}}(y,y_1,x) \\
\phantom{mmmmmm}\mathand \rho^\ast_{\textup{edge}}(y,y_2,x)\mathand y_1\neq y_2\bigr)\Bigr);
$
} 
\end{equation*}
this formula means that each vertex $y$ of the graph $\tuple{\maxJir {\lalg D} \cap \ideal x, \relE}$ is the endpoint of at least two edges. Clearly, in spite of its textual description,  $\rho^\ast_{\leq2}(x)$ below is (equivalent to) a first-order formula in  $\lang\ordsigma$:
\begin{equation*}
\parbox{8cm}{
$\rho^\ast_{\leq2}(x):= (\forall y)\bigl(\ormjir y x\then$ 
there are at most 
two 
\\
\phantom{mmmmmmm}elements $y'$ such that $ \rho^\ast_{\textup{edge}}(y,y',x)\bigr).
$
} 
\end{equation*}
The formula 
\begin{equation*}
\rho^\ast_{=2}(x):= \rho^\ast_{\exists2}(x)\mathand \rho^\ast_{\leq2}(x)
\end{equation*}
means that each vertex $y$ of the graph $\tuple{\maxJir {\lalg D} \cap \ideal x, \relE}$ is the endpoint of exactly two edges. 
Next, we claim that 
\begin{align}
\ormcyclic  x:= {}
&(\exists y)(\rmjir y x)\,\, \mathand \,\,   \rho^\ast_{=2}(x) \,\,\mathand
\label{alignBBvrbszNsjGa}
\\
&(\forall y)\Bigl(\rjir x y \then (\exists z)\bigl(\rmjir z x \mathand y\leq z \bigr) \Bigr)
\label{alignBBvrbszNsjGb}
\end{align}
which is a first-order formula in  $\lang\ordsigma$, is equivalent to $\rmcyclic x$ defined in \eqref{align:rSmbLf}. To show this, let $H:=\set{y: \ormjir y x}$. By the  \eqref{alignBBvrbszNsjGa} part of $\ormcyclic  x$,  $H\neq \emptyset$. If each vertex in a \emph{finite} graph is connected with exactly two other vertices by edges, then this graph is the disjoint union of circles of length at least three. Hence,  \eqref{alignBBvrbszNsjGa} implies that $H$ is a spanned multicircle of $\tuple{\maxJir {\lalg D}\cap\ideal x, \relE}$. Since every element in a finite lattice is the join of join-irreducible elements, part
 \eqref{alignBBvrbszNsjGb} gives that $x=\bigvee H$, and we conclude that $x$ is a multicyclic element. That is, $\ormcyclic  x$ implies that $\rmcyclic  x$, while the converse implication is trivial. Thus, $\ormcyclic  x$ is equivalent to $\rmcyclic  x$, as required.

Finally, consider the following formula, which is (clearly equivalent to) a first-order formula in $\lang\ordsigma$.
\begin{equation}
\parbox{8cm}{
$(\forall x)\Bigl( \ormcyclic x\then (\exists y)(\exists z)\Bigl(\ormmm y \mathand \ormmm z \mathand 
\\
\phantom{mmmm}
x= y\vee z \mathand \neg(\exists t)\bigl(   \ormjir t y \mathand \ormjir t z\bigr)\Bigr)\Bigr)
$,
}
\label{eq:mcrDsPWbpsz}
\end{equation}
where ``$\neg$'' stands for negation. The meaning of \eqref{eq:mcrDsPWbpsz} is that each multicyclic element is the join of two VW-elements such that the 
intersection of the principal ideals determined by the two joinands is disjoint from $\maxJir {\lalg D}$. 
Our aim is to show that, 
\begin{equation}
\parbox{8cm}{for a finite distributive lattice ${\lalg D}$ satisfying the Two-cover Property,   \eqref{eqpbx:nWpr} is equivalent to  \eqref{eq:mcrDsPWbpsz}.}
\label{eqpbx:slDrzzHlKhRpkbR} 
\end{equation}
Since cyclic elements are multicyclic, it is clear that  \eqref{eq:mcrDsPWbpsz}  implies \eqref{eqpbx:nWpr}.

To show the converse implication, assume that ${\lalg D}$ satisfies  \eqref{eqpbx:nWpr} and the Two-cover Property, and let $x$ be a multicyclic element of ${\lalg D}$. Then there is a spanned multicircle $H$ in the graph $\tuple{\maxJir {\lalg D}\cap\ideal x,\relE}$ such that $x=\bigvee H$. Let $H_1$, \dots, $H_k$ be the connected components (that is, the circles) of $H$, and let $x_1:=\bigvee H_1$, \dots, $x_k:=\bigvee H_k$. 
We have that, for $i\in\set{1,\dots,k}$,
\begin{equation}
H=\maxJir {\lalg D}\cap\ideal x\text{ and  } H_i=\maxJir {\lalg D}\cap \ideal x_i
\label{eq:pRmrHpGnB}
\end{equation}
since the ``$\subseteq$'' inclusions are trivial while the converse inclusions follow from \eqref{pbx:sZrkjNhZzafBrnsn}.
Since $H_i$ is clearly a spanned circle not only in $\tuple{\maxJir {\lalg D}\cap\ideal x,\relE}$ but also in $\tuple{\maxJir {\lalg D}\cap\ideal x_i,\relE}$,  $x_i$ is a cyclic element for $i\in\set{1,\dots, k}$. Since we have assumed \eqref{eqpbx:nWpr}, each $x_i$ is of the form $x_i=y_i\vee z_i$ with VW-elements $y_i,z_i\in D$ such that $\maxJir {\lalg D}\cap \ideal{y_i} \cap \ideal  {z_i}=\emptyset$. Let $Y_i:=H_i\cap\ideal y_i$ and $Z_i:=H_i\cap\ideal z_i$. Define $Y:=Y_1\cup\dots\cup Y_k$, $y:=y_1\vee\dots\vee y_k$,  $Z:=Z_1\cup\dots\cup Z_k$, and $z:=z_1\vee\dots\vee z_k$.
We claim that 
\begin{align}
&\text{for $i\in\set{1,\dots,k}$, \ $H_i$ is the disjoint union of $Y_i$ and $Z_i$,}
\label{align:wztGTgh}\\
&y_i=\bigvee Y_i,\text{ and }z_i=\bigvee Z_i,\quad\text{whence}\quad y:=\bigvee Y\text{ and }z=\bigvee Z.
\label{eq:sRjrzGncDmmr}
\end{align}
The equality $\maxJir {\lalg D}\cap \ideal{y_i}\cap\ideal{z_i}=\emptyset$ and $H_i\subseteq \maxJir {\lalg D}$ yield that $Y_i\cap Z_i=\emptyset$. 
If $p\in H_i$, then $p\leq \bigvee H_i=x_i=y_i\vee z_i$, and  \eqref{eqpbx:wsFjERnQvCz} gives that $p\in\ideal y_i$ or $p\in\ideal z_i$, and so $p\in Y_i\cup Z_i$. Hence $H_i\subseteq Y_i\cup Z_i$ and we conclude the validity of \eqref{align:wztGTgh}.
By the definition of $Y_i$, we have that $y_i\geq \bigvee Y_i$. Let $Y_i':=\maxJir {\lalg D}\cap\ideal y_i$ and observe that $Y_i'\supseteq Y_i$.  If $p\in Y_i'$, then $p\in\ideal x_i$ and \eqref{eq:pRmrHpGnB} give that $p\in H_i$, whence $p\in Y_i$. Hence $Y_i'=Y_i$. Since $y_i$ is a VW-element, \eqref{eq:WrjRslGzT} yields that $y_i=\bigvee Y_i'$.  
Combining this equality with $Y_i'=Y_i$, we conclude the first equality in  \eqref{eq:sRjrzGncDmmr}. The second equality in \eqref{eq:sRjrzGncDmmr} follows from $y_i$--$z_i$ symmetry while the rest of  \eqref{eq:sRjrzGncDmmr} is clear by the first two equalities and the definition of $y$, $Y$, $z$, and $Z$.  This proves \eqref{eq:sRjrzGncDmmr}.

Next, we prove that $y$ is a VW-element. We know that $y_1$ is a VW-element, whence the inequality part of \eqref{eq:WrjRslGzT} gives that $0<y_1\leq y$. Since 
\[Y=Y_1\cup\dots\cup Y_k\subseteq H_1\cup\dots\cup H_k\subseteq \maxJir {\lalg D},
\]
\eqref{pbx:sZrkjNhZzafBrnsn} and the third equality of \eqref{eq:sRjrzGncDmmr} give that 
\[y  \overset{\eqref{eq:sRjrzGncDmmr}} = \bigvee Y \overset{\eqref{pbx:sZrkjNhZzafBrnsn}}= \bigvee\bigl(\maxJir {\lalg D}\cap \ideal y\bigr), 
\]
that is, $y$ satisfies condition \eqref{eq:WrjRslGzT}. Now we turn our attention to  \eqref{align:slTmrBma}--\eqref{align:slTmrBmb}. Let $p\in \maxJir {\lalg D}\cap \ideal y$. 
Since $p\leq y=y_1\vee\dots\vee y_k$, \eqref{eqpbx:wsFjERnQvCz} yields an $i$ such that $p\leq y_i$.  Applying \eqref{align:slTmrBma}--\eqref{align:slTmrBmb} to $y_i$, there is a  V-set or W-set $F$ of $y_i$ such that $p\in F$.  Since $y_i\leq y$, we have that
$F$ is a V-set or W-set of $y$, too. To show its uniqueness, assume that $F'$ is also a  V-set or W-set of $y$ such that $p\in F'$. The elements of $F'$ are in $\maxJir {\lalg D}\cap\ideal y$, so they are in $\maxJir {\lalg D}\cap\ideal x$ since $y\leq x$. Applying \eqref{pbx:sZrkjNhZzafBrnsn} to $x=\bigvee H$, we obtain that $F'\subseteq H$. Let $q$ be another element of $F'$. 
Since $F'$ is a V-set or W-set of $y$, it follows from Definition~\ref{def:MsetMMset} that there is a $t\in\set{1,2,3}$ and there are elements $r_0=p$, $r_1$, \dots, $r_{t-1}$, and $r_t=q$ 
\begin{equation}
\text{in $\maxJir {\lalg D}\cap \ideal y$ such that $\Jir {\lalg D}\cap\ideal r_{j-1}\cap \ideal r_j\neq \emptyset$}
\label{eq:sZrnLm}
\end{equation}
for $j\in\set{1,\dots, t}$. Since $y\leq x$, we can replace $\ideal y$ in \eqref{eq:sZrnLm} by $\ideal x$. After this replacement, \eqref{eq:sZrnLm} and the meaning of $\relE$ yield that $p$ and $q$ are in the same component of the graph $\tuple{\maxJir {\lalg D}\cap\ideal x,\relE}$. But $H$ is a \emph{spanned} multicircle, whereby $p$ and $q$ are in the same component $H_s$ ($s\in\set{1,\dots,k}$) of $H$. We know that $p\in H_i$, whereby $q$ is in $H_i$ as well. This holds for all $q\in F'$, whence $F'\subseteq H_i$. 
Since $H_i=\maxJir {\lalg D}\cap \ideal x_i$ by \eqref{eq:pRmrHpGnB} and $\bigvee H_i=x_i$, 
we have that $r\leq x_i$ for every $r\in F'$. We also have that $r\leq y$ since $F'$ is a V-set of W-set of $y$. Hence, 
\begin{equation}
r\leq x_i\wedge y=x_i\wedge (y_1\vee\dots\vee y_k)=
(x_i\wedge y_1)\vee\dots (x_i\wedge y_k).
\label{eq:wzKtJsDb}
\end{equation}
If we had that $r\leq y_j$ for some $j\in\set{1,\dots,k}\setminus\set i$, then we would have that $r\in H_i$ by  $F'\subseteq H_i$
and $r\in H_j$ by $r\leq y_j\leq x_j$ and \eqref{eq:pRmrHpGnB}, contradicting $H_i\cap H_j=\emptyset$. Hence, $r\not\leq y_j$ if $j\neq i$, whereby \eqref{eqpbx:wsFjERnQvCz} and \eqref{eq:wzKtJsDb} imply that
$r\leq x_i\wedge y_i=y_i$. So every $r\in F'$ belongs to $\maxJir {\lalg D}\cap \ideal y_i$ and it follows that $F'$ is a V-set or W-set of $y_i$. But $y_i$ is a VW-element, whereby there is only one V-set or W-set containing $p$. Thus, $F'=F$, implying the uniqueness of the V-set or W-set of $y$ containing $p$. We have shown that $y$ is a VW-element.

Since the role of $y$ and $z$ is symmetric, $z$ is a VW-element, too. Clearly,
\begin{equation}
\begin{aligned}
y\vee z&=(y_1\vee\dots\vee y_k) \vee (z_1\vee\dots\vee z_k) 
\cr
&= (y_1\vee z_1)\vee \dots \vee (y_k\vee z_k)= x_1\vee\dots\vee x_k=x.
\end{aligned}
\label{eq:szHjRwhBr}
\end{equation}
We claim that 
\begin{equation}
\maxJir {\lalg D}\cap \ideal y\cap \ideal z=\emptyset.
\label{eq:szbLwGcNrcPt}
\end{equation}
For the sake of contradiction, suppose that \eqref{eq:szbLwGcNrcPt} fails. 
Then we can pick an  $s\in \maxJir {\lalg D}$ such that $s\leq y$ and $s\leq z$. Observe that $Y_i\subseteq H_i\subseteq H\subseteq \maxJir {\lalg D}$. 
Since $s\leq y=y_1\vee\dots\vee y_k$, 
\eqref{eqpbx:wsFjERnQvCz} yields a subscript $i\in\set{1,\dots,k}$ such that $s\leq y_i$. Using \eqref{pbx:sZrkjNhZzafBrnsn}, $Y_i\subseteq \maxJir {\lalg D}$, and the first equality of  \eqref{eq:sRjrzGncDmmr},  we obtain that $s\in Y_i$. Similarly, we obtain a subscript $j\in\set{1,\dots, k}$ such that $s\in Z_j$. If we had that $i=j$, then \eqref{eq:sRjrzGncDmmr} would lead to $s\leq \bigvee Y_i\wedge \bigvee Z_i=y_i\wedge z_i$, contradicting that $y_i$ and $z_i$ were chosen for the cyclic element $x_i$ according to \eqref{eqpbx:nWpr}. Consequently, $i\neq j$. Hence $s\in Y_i$, $s\in Z_j$, $Y_i\subseteq H_i$, and $Z_j\subseteq H_j$  lead to $s\in H_i\cap H_j$, contradicting $H_i\cap H_j=\emptyset$. Thus, \eqref{eq:szbLwGcNrcPt} holds.  

Finally, \eqref{eq:szHjRwhBr}, \eqref{eq:szbLwGcNrcPt}, and the fact that $y$ and $z$ are VW-elements imply that $x$ satisfies property \eqref{eq:mcrDsPWbpsz}. We have shown that \eqref{eqpbx:nWpr} is equivalent to  \eqref{eq:mcrDsPWbpsz}. This proves part (C) of Theorem~\ref{thm:newproptY}. The proof of Theorem~\ref{thm:newproptY} is complete.
\end{proof}

Let $2\leq n\in\Nplu$. By the well-known structure theorem of finite distributive lattices, see Gr\"atzer~\cite[Theorem II.1.9]{GrGLT} for example, there is a unique distributive lattice $\lalg D_n$ with $\aJir {\lalg D_n}\cong{\alg K_n}$. (For $\alg K_8$, see Figure~\ref{figK8}.)
As in the proof of Theorem~\ref{thmnonax}, $J_0$ and $J_1$ still denote $\set{2,4,6,8,\dots}$ and $\set{3,5,7,9,\dots}$, respectively. The main idea of the proof of Theorem~\ref{thmnonax} was to show that $\set{\alg K_n: n\in J_0}$ 
cannot be distinguished from $\set{\alg K_n: n\in J_1}$ by a first-order formula.  
However, the Decomposable Cyclic Elements Property, which is equivalent to the first-order formula \eqref{eq:mcrDsPWbpsz}, holds in $\set{\lalg D_n: n\in J_0}$ but fails in $\set{\lalg D_n: n\in J_1}$. Thus we conclude the following.

\begin{remark}\label{rem:szplMndHcKzgTrs}
The method of Section~\ref{sectionlatt} is not appropriate to decide whether the class $\{\Con {\lalg L}: {\lalg L}$ is a slim semimodular lattice$\}$ is definable by finitely many axioms among finite lattices. This question remains an open problem.
\end{remark}

As a by-product of the  proof of Theorem~\ref{thmnonax}, see \eqref{eqpbx:swGjwmknHpRc}, we have the following.

\begin{remark}\label{rem:nWhWrlNsW}
There is no first-order formula in $\lang\ordsigma$ such that the poset $\aJir {\lalg D}$ of join-irreducible elements of a finite distributive lattice $\lalg D$  satisfies this formula if and only if $\lalg D$ satisfies the Bipartite Maximal Elements Property; see Definition~\ref{def:bzwBNrtmG}. 
\end{remark}

Before the present paper, seven properties of the congruence lattices $\Con {\lalg L}$ of slim semimodular lattices ${\lalg L}$ have been known. There are two in Gr\"atzer~\cite{gr:confork} and \cite{grVIII}, four in Cz\'edli~\cite{czglamps}, and one in Cz\'edli and Gr\"atzer~\cite{czggg3p3c}. In six out of these seven cases, the corresponding property of $\aJir{\Con {\lalg L}}$ is clearly given (or can trivially be given) by a first-order formula of $\lang\ordsigma$. Combining this first-order formula with 
$\rjir y 1$ from \eqref{align:rSmbLa}, it is easy to see that each of the six properties can be given by a single first-order condition  tailored to $\Con {\lalg L}$. The seventh property, proved in Cz\'edli~\cite{czglamps}, is the Bipartite Maximal Elements Property; see Definition~\ref{def:bzwBNrtmG} in the present paper. Now the following remark follows from Lemma~\ref{lemma:twocoverP}, parts (B) and (C) of Theorem~\ref{thm:newproptY}, and  the fact that finitely many formulas can be replaced by their conjunction.

\begin{remark}\label{rem:sNldGsxm}
There exists a single first-order formula in $\lang\ordsigma$ such that a finite (distributive) lattice ${\lalg D}$ satisfies this formula if and only if all the seven known properties of the congruence lattices of slim semimodular lattices hold in ${\lalg D}$.
\end{remark}

\subsection{Acknowledgment} I am grateful to Mike Behrisch, Manuel Bodirsky, Brian Davey, and Marcel Jackson for their bibliographic comments. I express my additional gratitude to Mike Behrisch and  Marcel Jackson for their comments on Ehrenfeucht-Fra\"\i ss\'e Games.

\end{document}